\documentclass{amsart}
\usepackage{amssymb,amsthm}
\usepackage{hyperref}

\newtheorem{theorem}{Theorem}[section]
\newtheorem{lemma}[theorem]{Lemma}

\newtheorem{conjecture}[theorem]{Conjecture}
\newtheorem{proposition}[theorem]{Proposition}
\newtheorem{example}[theorem]{Example}
\theoremstyle{definition}
\newtheorem{definition}[theorem]{Definition}

\newcommand{\C}{\mathrm{C}}
\newcommand{\D}{\mathrm{D}}

\newcommand{\FF}{\mathbb F}

\newcommand{\Cay}{\mathop{\mathrm{Cay}}}
\newcommand{\soc}{\mathop{\mathrm{soc}}}
\newcommand{\Aut}{\mathop{\mathrm{Aut}}}
\newcommand{\Sym}{\mathop{\mathrm{Sym}}}
\newcommand{\Alt}{\mathop{\mathrm{Alt}}}
\newcommand{\PSL}{\mathop{\mathrm{PSL}}}
\newcommand{\PGL}{\mathop{\mathrm{PGL}}}
\newcommand{\SL}{\mathop{\mathrm{SL}}}

\renewcommand{\wr}{\mathop{\rm wr}}

\newcommand{\vGa}{{\Gamma}}

\newcommand{\UDRR}{\mathrm{UDRR}}

\newcommand{\UAGRR}{\mathrm{USmall}}

\newcommand{\UCDN}{\mathrm{UCDN}}
\newcommand{\UCGN}{\mathrm{UCGN}}

\def\Z#1{{\bf Z}(#1)}
\def\cent#1#2{{\bf C}_{#1}(#2)}
\def\nor#1#2{{\bf N}_{#1}(#2)}
\newcommand{\norml}{\trianglelefteq}

\title{Cayley graphs on abelian groups}

\author[E. Dobson]{Edward Dobson}
\address{Edward Dobson, Department of Mathematics and Statistics,   \newline
Mississippi State University, PO Drawer MA, Mississippi State, MS 39762, USA.\newline
\indent Also affiliated with : University of Primorska, In\v{s}titut Andrej Maru\v{s}ic, \newline
\indent Muzejski trg 2, 6000 Koper, Slovenia.}
\email{dobson@math.msstate.edu}

\author[P. Spiga]{Pablo Spiga}
\address{Pablo Spiga, Departimento di Matematica Pura e Applicata,\newline
 University of Milano-Bicocca, Via Cozzi 53, 20126 Milano, Italy.}
\email{pablo.spiga@unimib.it}

\author[G. Verret]{Gabriel Verret}
\address{Gabriel Verret, Centre for Mathematics of Symmetry and Computation, \newline
\indent School of Mathematics and Statistics, The University of Western Australia, \newline
\indent 35 Stirling Highway, Crawley, WA 6009, Australia. \newline
\indent Also affiliated with : University of Primorska, FAMNIT, \newline
\indent Glagolja\v{s}ka 8, 6000 Koper, Slovenia.}

\email{gabriel.verret@uwa.edu.au}

\thanks{The third author is supported by UWA as part of the Australian Research Council grant DE130101001. Address correspondence to Pablo Spiga: pablo.spiga@unimib.it.}



\begin{document}

\begin{abstract}
Let $A$ be an abelian group and let $\iota$ be the automorphism of $A$ defined by $\iota:a\mapsto a^{-1}$. A Cayley graph $\vGa=\Cay(A,S)$ is said to have an automorphism group \emph{as small as possible} if $\Aut(\Gamma)= A\rtimes\langle\iota\rangle$. In this paper, we show that almost all Cayley graphs on abelian groups have automorphism group as small as possible, proving a conjecture of Babai and Godsil.
\end{abstract}

\maketitle
\baselineskip=1.3\normalbaselineskip

\section{Introduction}
All digraphs and groups considered in this paper are finite. By a {\em digraph} $\vGa$, we mean an ordered pair $(\mathcal{V},\mathcal{A})$ where the \emph{vertex-set} $\mathcal{V}$ is a finite non-empty set and the {\em arc-set} $\mathcal{A}$ is a binary relation on $\mathcal{V}$. The elements of $\mathcal{V}$ and $\mathcal{A}$ are called \emph{vertices} and \emph{arcs} of $\vGa$, respectively. The digraph $\vGa$ is called a {\em graph} when the relation $\mathcal{A}$ is symmetric. An \emph{automorphism} of $\vGa$ is a permutation of $\mathcal{V}$ which preserves the relation $\mathcal{A}$.

Let $G$ be a group and let $S$ be a subset of $G$. The \emph{Cayley digraph} on $G$ with connection set $S$, denoted $\Cay(G,S)$, is the digraph with vertex-set $G$ and with $(g,h)$ being an arc if and only if $gh^{-1}\in S$. Note that we do not require our Cayley digraphs to be connected and that they may have loops. It is an obvious observation that $\Cay(G,S)$ is a graph if and only if $S$ is inverse-closed, in which case it is called a \emph{Cayley graph}. It is also easy to check that $G$ acts regularly as a group of automorphisms of $\Cay(G,S)$ by right multiplication.

When studying a Cayley digraph $\Cay(G,S)$, a very important question is to determine whether $G$ is in fact the full automorphism group. When it is, $\Cay(G,S)$ is called a \emph{DRR} (for digraphical regular representation). A DRR which is a graph is called a \emph{GRR} (for graphical regular representation).

DRRs and GRRs have been widely studied. The most natural question is the ``GRR problem'': which groups admit GRRs? The answer to this question was completed by Godsil~\cite{Godsil}, after a long series of partial results by various authors (see~\cite{Hetzel,Imrich,NowWat} for example). The equivalent problem for digraphs was solved by Babai~\cite{Ba} (curiously, the ``DRR problem'' was mainly considered after the GRR problem had been solved). In the course of working on these and related problems, Babai and Godsil made the following conjecture~\cite{BaGo}.

\begin{conjecture}\label{digraphmainconjecture}
Let $G$ be a group of order $n$. The proportion of subsets $S$ of $G$ such that $\Cay(G,S)$ is a DRR goes to $1$ as $n\to\infty$.
\end{conjecture}

In other words, ``almost all Cayley digraphs are DRRs''. Godsil showed that Conjecture~\ref{digraphmainconjecture} holds if $G$ is a $p$-group with no homomorphism onto $\C_p \wr \C_p$~\cite{Go2}, and Babai and Godsil extended this to verify the conjecture in the case that $G$ is nilpotent of odd order~\cite[Theorem 2.2]{BaGo}.  One of the results of this paper is a proof of Conjecture \ref{digraphmainconjecture} when $G$ is an abelian group.

\begin{theorem}\label{th:main1}
Let $A$ be an abelian group of order $n$.  The proportion of subsets $S$ of $A$ such that $\Cay(A,S)$ is a DRR goes to $1$ as $n\to\infty$.
\end{theorem}

It is not possible to prove a directly analogous result for inverse-closed subsets and GRRs, for simple reasons which we now explain.

Let $A$ be an abelian group and let $\iota$ be the automorphism of $A$ defined by $\iota:a\mapsto a^{-1}$ for every $a\in A$. It is not hard to see that every Cayley graph on $A$ admits $A\rtimes\langle\iota\rangle$ as a group of automorphisms. On the other hand, if $A$ has exponent greater than $2$ then $\iota\neq 1$ and $A\rtimes\langle\iota\rangle>A$, and hence no Cayley graph on $A$ is a GRR.

Similarly, a generalized dicyclic group  also admits a non-trivial automorphism which maps every element either to itself or to its inverse  (see~\cite{Nowitz1968}) and hence generalized dicyclic groups form another infinite family of groups which do not admit GRRs. It is believed that these two families are the only obstructions to Conjecture \ref{digraphmainconjecture} holding for graphs. More precisely, Babai, Godsil, Imrich and Lov\'asz made the following conjecture~\cite[Conjecture 2.1]{BaGo}.

\begin{conjecture}\label{graphmainconjecture}
Let $G$ be a group of order $n$ which is neither generalized dicyclic nor abelian of exponent greater than $2$. The proportion of inverse-closed subsets $S$ of $G$ such that $\Cay(G,S)$ is a GRR goes to $1$ as $n\to\infty$.
\end{conjecture}

As in the digraph case, Godsil showed that Conjecture~\ref{graphmainconjecture} holds if $G$ is a $p$-group with no homomorphism onto $\C_p \wr \C_p$~\cite{Go2} while Babai and Godsil verified Conjecture \ref{graphmainconjecture} in the case that $G$ is nilpotent of odd order~\cite[Theorem 2.2]{BaGo}.

If $A$ is abelian of exponent greater than $2$, the preceding observations make it natural to conjecture that ``almost all Cayley graphs of $A$ have automorphism group as small as possible (namely $A\rtimes\langle\iota\rangle$)". This conjecture was made by Babai and Godsil \cite[Remark~4.2]{BaGo}.

\begin{conjecture}\label{graphsecondconjecture}
Let $A$ be an abelian group of order $n$. The proportion of inverse-closed subsets $S$ of $A$ such that $\Aut(\Cay(A,S)) = A\rtimes\langle\iota\rangle$ goes to $1$ as $n\to\infty$.
\end{conjecture}

Babai and Godsil verified Conjecture~\ref{graphsecondconjecture} when $A$ has order congruent to $3\pmod 4$~\cite[Theorem 5.3]{BaGo}.  Additionally, Godsil pointed out that~\cite[Corollary 4.4]{Go2} could be used to show that Conjecture~\ref{graphsecondconjecture} is true if $A$ has odd prime-power order~\cite[Page 253]{Go2}. This fact was actually proved by the first author using different ideas~\cite{Dobson2010b}. A translation of results proven using Schur rings in \cite{EvdokimovP2002, LeungM1996, LeungM1998} into group-theoretic language gives strong constraints on transitive permutation groups containing a regular cyclic subgroup~\cite[Theorem 1.2]{Li2005}. Using this translation, Bhoumik, Morris and the first author recently verified Conjecture~\ref{graphsecondconjecture} for $A$ a cyclic group~\cite{BhoumikDM2013}.  In this paper, we extend these results and prove Conjecture~\ref{graphsecondconjecture}.

\begin{theorem}\label{th:main2}
Let $A$ be an abelian group of order $n$. The proportion of inverse-closed subsets $S$ of $A$ such that $\Aut(\Cay(A,S)) = A\rtimes\langle\iota\rangle$ goes to $1$ as $n\to\infty$.
\end{theorem}

We stated Theorems~\ref{th:main1} and~\ref{th:main2} in this way for simplicity but, in fact, we prove the following more explicit versions.

\begin{theorem}\label{th:epsilon1}
Let $A$ be an abelian group of order $n$.  Then the number of subsets $S$ such that $\Cay(A,S)$ is not a DRR is at most $2^{3n/4+2(\log_2(n))^2+1}$.
\end{theorem}

\begin{theorem}\label{th:epsilon2}
Let $A$ be an abelian group of order $n$ and let $m$ be the number of elements of order at most $2$ of $A$.  Then the number of inverse-closed subsets $S$  with $\Aut(\Cay(A,S))>A\rtimes\langle\iota\rangle$ is at most $2^{m/2+11n/24+2(\log_2(n))^2+2}$.
\end{theorem}

An analogue of Theorem~\ref{th:main2} for generalised dicyclic groups was recently proved by Morris and the last two authors~\cite{Morris}. These results also provide supporting evidence for two conjectures of Xu.  A Cayley (di)graph $\Gamma$ of $G$ is said to be a {\it normal} Cayley (di)graph of $G$ if the regular representation of $G$ is normal in $\Aut(\Gamma)$.  Xu conjectured that almost all Cayley (di)graphs of $G$ are normal Cayley (di)graphs of $G$ (in the undirected case, there is a known exceptional family of groups which must be excluded). See~\cite[Conjecture 1]{Xu1998} for the precise formulation of these conjectures. In fact, it follows from Lemma~\ref{lemma:normalizeddigraphcounting} that Xu's digraph conjecture is equivalent to Conjecture~\ref{digraphmainconjecture}.  Our results support these conjectures as any Cayley (di)graph on $G$ that has automorphism group as small as possible is a normal Cayley (di)graph of $G$.

\subsection{Structure of the paper}

We now give a brief summary of the rest of the paper. Section~\ref{sec:prelim} contains some preliminary results about permutation groups which are needed for  Section~\ref{sec:abelianregular}.  In Section~\ref{sec:abelianregular}, we prove two theorems about permutation groups $G$ containing an abelian regular subgroup $A$ such that the normalizer $\nor G A$ of $A$ in $G$ is either $A$ (see Theorem~\ref{thm:main}) or $A\rtimes\langle\iota\rangle$ (see Theorem~\ref{thm:main2}) and with $\nor G A$ maximal in $G$. (There is an extra technical condition in the statement of Theorem~\ref{thm:main2}.) In both cases we give a fairly detailed description of the structure of $G$.

In Section~\ref{sec:AppDigraph}, we extend our results about permutation groups from Section~\ref{sec:abelianregular} to prove some structural results about Cayley (di)graph on abelian groups. In loose terms, we show that a Cayley graph over an abelian group $A$ is either a generalized wreath graph (see Definition~\ref{def:semiwreath}), or admits a very specific decomposition as a direct product, or admits a non-trivial automorphism different from $\iota$ normalizing $A$. Consequences of these results (Theorems \ref{thm:graphmain} and \ref{thm:graphmain2}) can be considered generalizations of~\cite[Theorem 1.2]{Li2005} in the more general context of abelian groups.

In Section~\ref{sec:EnumerationDigraphs}, we apply the results from Section~\ref{sec:AppDigraph} to prove Theorems~\ref{th:epsilon1} and ~\ref{th:epsilon2}, which imply Theorems~\ref{th:main1}  and~\ref{th:main2}. Finally, in Section~\ref{sec:unlabelled}, we show that the corresponding version of our results for unlabeled graphs easily follows.

\subsection{A few comments}

In light of Theorem~\ref{thm:main}, we feel that it might be interesting in the future to drop the condition of maximality, that is to study transitive permutation groups containing a self-normalizing abelian regular subgroup (in other words, a regular abelian Carter subgroup). Spurred by this investigation, Jabara and the second author recently proved that these groups are in fact solvable~\cite{JabSpiga}. Together with  Casolo, they also proved an upper bound on the Fitting height of such a group in terms of the Fitting height and the derived length of a point-stabilizer (and some extra mild hypothesis)~\cite{CJS}. We think that a classification (in a very broad sense) of these groups would be quite interesting, although perhaps a little optimistic.

The condition ``$\nor G A=A$'' is very natural in the context of enumeration of Cayley (di)graphs. Indeed, if $A$ is a regular subgroup of a permutation group $G$ and $\nor G A>A$, then $G$ contains an element acting as a non-trivial automorphism on $A$ and upper bounds on the frequency of this occurence can often be obtained (see Lemma~\ref{lemma:normalizeddigraphcounting} for example).

The hypothesis ``$\nor G A=A$'' is thus often a critical one. For example, it was used by Godsil in a crucial step of the proof of~\cite[Theorem~$3.6$]{Go2}, allowing him to use a deep transfer-theoretic result of Yoshida~\cite[Theorem~$4.3$]{Yoshida}. It was also used by Poto\v{c}nik and the second and third authors to enumerate Cayley graphs and GRRs of a fixed valency~\cite{Enumeration}.

\section{Preliminaries}
\label{sec:prelim}

In this section, we prove two results which will be used in Section~\ref{sec:abelianregular}. We could not find a reference for the following result in the form tailored to our needs, thus we include a proof.

\begin{lemma}\label{lemma:Fr}
Let $G$ be a primitive group with an abelian point-stabilizer. Then the socle of $G$ is a regular elementary abelian $p$-group for some prime $p$, and the point-stabilizers of $G$ are cyclic of order coprime to $p$.
\end{lemma}
\begin{proof}
Let $A$ be the stabilizer of a point in $G$. If $A=1$, then $G$ is a cyclic group of prime order. Suppose that $A>1$. Let $g\in G\setminus A$. By the maximality of $A$ in $G$, it follows that $\langle A,A^g\rangle=G$. Now $A\cap A^g$ is centralized by $A$ and $A^g$ and hence by $G$. It follows that $A\cap A^g=1$. We have shown that $A\cap A^g=1$ for every $g\in G\setminus A$, from which it follows that $G$ is a Frobenius group with complement $A$. Let $N$ be the Frobenius kernel. Observe that $N$ is regular. Since $N$ is nilpotent and $G$ is primitive, it follows that $N$ is elementary abelian. Since $G$ is primitive, $A$ acts irreducibly as a linear group on $N$. From Schur's lemma we deduce that $A$ is cyclic of order coprime to $|N|$.
\end{proof}

We say that a group $B$ is a {\em generalized dihedral group} on $A$, if $A$ is an abelian subgroup of index $2$ in $B$ and there exists an involution $\iota\in B\setminus A$ with $a^\iota=a^{-1}$ for every $a\in A$. Note that, in this case, $a^x=a^{-1}$ for every $a\in A$ and every $x\in B\setminus A$. We denote by $\C_n$ the cyclic group of order $n$ and by $\D_n$ the dihedral group of order $2n$. For terminology regarding the types of primitive groups, we refer to~\cite{ONAN}.

\begin{proposition}\label{prop:ch}
Let $G$ be a primitive group such that a point-stabilizer $B$ is a generalized dihedral group on $A$ and such that $G$ contains a subgroup $L$ with $G=LB$ and $|L\cap B|\leq 2$. Then one of the following holds:
\begin{itemize}
\item $G$ is of affine type,
\item $G\cong\PGL(2,q)$ for some prime power $q\geq 4$, $B\cong \D_{q+1}$, $A\cong\C_{q+1}$, $|B\cap L|=2$ and $G$ in its action on the right cosets of $L$ is $2$-transitive,
\item $G\cong \PGL(2,q)$ for some prime power $q\geq 7$ with $q\equiv 3\pmod 4$, $B\cong \D_{q+1}$, $A\cong \C_{q+1}$ and $|B\cap L|=1$,
\item $G\cong \PSL(2,q)$ for some prime power $q\geq 11$ with $q\equiv 3\pmod 4$, $B\cong \D_{(q+1)/2}$ and $|B\cap L|=1$.
\end{itemize}
\end{proposition}

\begin{proof}
We assume that $G$ is not of affine type. The finite primitive groups with a solvable point-stabilizer are classified in~\cite{LiZ2011}. From~\cite[Theorem~$1.1$]{LiZ2011} we see that $G$ is of almost simple or product action type.

Suppose that $G$ is of almost simple type. It follows from~\cite[Theorem~$1.1$~(ii)]{LiZ2011} that $G$ contains a normal subgroup $G_0$ which is minimal with respect to the property that $B_0=B\cap G_0$ is maximal in $G_0$ and $|G:G_0|=|B:B_0|$. Moreover, $(G_0,B_0)$ is one of the pairs in~\cite[Tables~14--20]{LiZ2011}. Since $B$ is a generalized dihedral group, $B_0$ is either abelian or a generalized dihedral group. Let $T$ be the socle of $G$. A meticulous analysis of the pairs in~\cite[Tables~14--20]{LiZ2011} shows that $(T,G_0,B_0)$ must be one of the triples in Table~\ref{table}. In particular, $B_0$ is a dihedral group and $|B_0:G_0\cap A|=2$.

\begin{table}[hh]
\begin{tabular}{|c|c|c|c|}\hline
$T$&$G_0$&$B_0$&Comments\\\hline
${^2}B_2(q)$&${^2}B_2(q)$&$\D_{q-1}$&\\
$\PSL(2,q)$&$\PSL(2,q)$&$\D_{(q-1)/(2,q-1)}$&$q\neq 5,7,9,11$\\
$\PSL(2,q)$&$\PSL(2,q)$&$\D_{(q+1)/(2,q-1)}$&$q\neq 7,9$\\
$\PSL(2,7)$&$\PGL(2,7)$&$\D_6$, $\D_8$&\\
$\PSL(2,11)$&$\PGL(2,11)$&$\D_{10}$&\\\hline
\end{tabular}
\caption{}\label{table}
\end{table}
 We consider each line of Table~\ref{table} on  a case-by-case basis.
 Note that $T$ cannot be a Suzuki group ${^2}B_2(q)$ because an almost simple group $G$ with such a socle does not admit a factorization with $G=LB$, $|L\cap B|\leq 2$, and $B\neq 1\neq L$, see~\cite[Theorem~B]{LiebeckPS1990}. Therefore, $T=\PSL(2,q)$ for some prime power $q$.

Suppose that $B_0=\D_{(q-1)/(2,q-1)}$ with $q\neq 5,7,9,11$. Then, according to Table~\ref{table}, $G_0=T$ and $|T\cap A|=(q-1)/(2,q-1)$. It follows from~\cite[Table~$1$]{LiebeckPS1990} that the factorization $G=BL$ gives rise to the factorization $T=(T\cap B)(T\cap L)$. Since $|B\cap L|\leq 2$ and $|T|=q(q^2-1)/(2,q-1)$, we obtain $|T\cap L|=|T||(T\cap B)\cap (T\cap L)|/|T\cap B|\geq q(q+1)/2$. A quick look at the maximal subgroups of $\PSL(2,q)$~(\cite[Theorem~6.17]{Suzuki1982}) reveals that $T$ has a subgroup $T\cap L$ of such large order only when $q=2^\ell$ and $T\cap L$ is a Borel subgroup of $T$, that is, $|T\cap L|=q(q-1)$. Now, $q(q^2-1)=|T|=|(T\cap B)(T\cap L)|$ divides $|T\cap B||T\cap L|=2q(q-1)^2$ and hence $q+1$ divides $2(q-1)$, which is impossible for $q>3$.

Suppose now that $B_0=\D_{(q+1)/(2,q-1)}$ (with $q\neq 7,9$) and hence $G_0=T$. Let $A_0=B_0\cap A$. The group $A_0$ is cyclic of order $(q+1)/(2,q-1)$. In other words, $A_0$ is a maximal non-split torus of $T$.  Let $\lambda$ be a generator of the cyclic group $\mathbb{F}_{q^2}^*$. Now, under the isomorphism $\mathbb{F}_{q}^2\cong \mathbb{F}_{q^2}$, the group $A_0$ corresponds to $\langle\lambda^{(2,q-1)}\rangle/\langle\lambda^{q+1}\rangle$, and $\nor{{\mathrm{P}\Gamma\mathrm{L}(2,q)}}{A_0}$ corresponds to $(\langle\lambda\rangle/\langle\lambda^{q+1}\rangle)\rtimes \langle w,F\rangle$, where $w$ is the generator of the Weyl group acting by $w:\lambda\mapsto \lambda^{-1}$, and where $F$ is the Galois group of $\mathbb{F}_q$ over its ground field. Write $q=p^f$, with $p$ a prime and $f\geq 1$. Thus $F$ is cyclic of order $f$ generated by $\sigma:\lambda\to\lambda^p$. We show that no non-trivial element $w^\varepsilon\sigma^e$ of $\langle w,F\rangle$ centralizes $\langle \lambda^{(2,q-1)}\rangle/\langle\lambda^{q+1}\rangle$. If $\varepsilon(2,q-1)p^e\equiv (2,q-1)\pmod{q+1}$ for some $\varepsilon\in \{-1,1\}$ and $0\leq e<f$, then $q+1=p^f+1$ divides $(2,q-1)(\varepsilon p^e-1)$ and hence $\varepsilon=1$ and $e=0$. This shows that $\cent {{\mathrm{P}\Gamma\mathrm{L}(2,q)}}{A_0}$ is cyclic of order $q+1$ and is contained in $\PGL(2,q)$. As $B=B_0A$, $A$ centralizes $A_0$ and $G=TB=T(B_0A)=(TB_0)A=TA$, we get $G\leq \PGL(2,q)$. If $G=\PGL(2,q)$, then $B=\D_{q+1}$, $A\cong C_{q+1}$ and $|L|\in \{q(q-1)/2,q(q-1)\}$. If $|L|=q(q-1)$, then $L$ is a Borel subgroup of $G$  and hence the action of $G$ on the right cosets of $L$ is permutation equivalent to the action of $G$ on the points of the projective line, which is $2$-transitive, and thus the result follows. If $|L|=q(q-1)/2$, then $|B\cap L|=1$ and $G=BL$ is an exact factorization. It follows from~\cite[Table~$1$]{LiebeckPS1990} that $q\equiv 3\pmod 4$ and the result follows. Suppose now that $G< \PGL(2,q)$: then $q$ is odd, $G=T$, $B=\D_{(q+1)/2}$ and $A=\C_{(q+1)/2}$. As $|B\cap L|\leq 2$, we have  $|L|=q(q-1)$ or $|L|=q(q-1)/2$. Another quick look at the maximal subgroups of $\PSL(2,q)$ again reveals that $L$ is a Borel subgroup of $T$ and hence has order $q(q-1)/2$. In particular, $B\cap L=1$.  As above, it follows from~\cite[Table~$1$]{LiebeckPS1990} that $q\equiv 3\pmod 4$ and the result follows.

Suppose that $G_0=\PGL(2,q)$ and hence, according to Table~\ref{table}, $q\in \{7,11\}$. In this case, $q$ is prime and hence $G=G_0$ and $B=B_0$. Suppose that $q=7$. If $B=\D_8$, then $A=\C_8$. If $B\cap L=1$, then the result follows. If $|B\cap L|=2$, then $|L|=42$ and $L$ is a Borel subgroup of $G$.  In particular, the action of $G$ on the right cosets of $L$ is permutation equivalent to the action of $G$ on the points of the projective line, which is $2$-transitive, and hence the result follows.  If $B=\D_6$, then $B$ has order $12$ and hence $L$ has index $6$ or $12$ in $G$, but $\PGL(2,7)$ does not have a subgroup of index $6$ or $12$. Suppose that $q=11$ and hence $B=\D_{10}$. It follows that $A$ has order $10$.  As $G = LB$ and $|L\cap B|\leq 2$, we have $|G:L|\in \{10,20\}$. If $|G:L|=10$, we may view $L$ as a point-stabilizer of the transitive action of $G$ on the $10$ cosets of $L$.  Since $\Sym(10)$ contains no element of order $11$, every element of order $11$ in $G$ must be contained in the kernel of this action.  This implies that the kernel of this action contains $\PSL(2,11)$, which contradicts the fact that the action is transitive. Thus $|G:L|=20$ and $|L|=66$, but $G$ has no subgroups of order $66$.

Finally, suppose that $G$ is of product action type. In particular $N\norml G\leq G_1\wr \Sym(m)$, with $m\geq 2$, with $G_1$ an almost simple group with socle $T$ and with $N=\soc(G)\cong T^m$. Let $N=T_1\times \cdots \times T_m$ with $T_i\cong T$ for every $i\in \{1,\ldots,m\}$. For every $i\in \{1,\ldots,m\}$, let $B_i=B\cap T_i$. From the structure of primitive groups of product action type~\cite{ONAN}, we have $B\cap N=B_1\times \cdots \times B_m$ with $|B_1|=\cdots =|B_m|>1$. As $B$ is maximal in $G$, we have $G=NB$ and hence $B$ must act transitively on $\{T_1,\ldots,T_m\}$. It follows that $B$ also acts transitively on $\{B_1,\ldots,B_m\}$ and, since $A\unlhd B$, also on $\{(B_1\cap A),\ldots,(B_m\cap A)\}$. However, as $B$ is a generalized dihedral group, $B$ normalizes every subgroup of $A$. Since $m\geq 2$, it follows that $B_1\cap A=\cdots =B_m\cap A=1$.  As $|B:A|\leq 2$, we have $|B_i|=2$ for every $i$ and hence $B\cap N$ is an elementary abelian $2$-group. Since $B\cap N\triangleleft B$ and since $B$ is a maximal subgroup of $G$, we get $B=\nor G{{B\cap N}}$ from which we obtain $B\cap N=\nor N {B\cap N}$. It follows that $B_i=\nor {T_i}{B_i}$. Since $B_i$ is self-normalizing, it is a Sylow $2$-subgroup of $T_i$.  As $\vert B_i\vert = 2$, it follows from Burnside's $p$-complement Theorem (see~\cite[7.2.1]{KurStell} for example) that $T_i$ has a normal $2$-complement, a contradiction.
\end{proof}

\section{Abelian regular subgroups with small normalizers}
\label{sec:abelianregular}
The first result of this section (Theorem~\ref{thm:main}) deals with permutation groups containing a self-normalizing abelian regular subgroup. We start with an example, which will hopefully help the reader to follow the proof of  Theorem~\ref{thm:main}.

\begin{example}\label{ex1}{\rm
Let $p$ be a prime, let $S$ be an abelian group and let $W$ be a non-trivial irreducible $\mathbb{F}_pS$-module over the field $\mathbb{F}_p$ of order $p$. Let $Q$ be a non-trivial abelian $p$-group, let $P=
W\times Q$ and let $S$ act on $P$ as a group of automorphisms by centralizing $Q$. Let $A=Q\times S$ and $G=P\rtimes S$.

Fix $q$ an element of $Q$ of order $p$ and let $w_1,\ldots,w_\ell$ be a basis of $W$ as an $\mathbb{F}_p$-vector space. Let $G_1=\langle qw_1,\ldots,qw_\ell\rangle$ and let $\Omega$ be the set of right cosets of $G_1$ in $G$. Clearly, $P=G_1\times Q$, $G=G_1A$ and $G_1\cap A=1$. In particular, the abelian group $A$ acts regularly on $\Omega$.

Let $w\in \nor W A$. For every $a\in A$, we have $a^{w}\in A$. Since $a^w=w^{-1}aw=w^{-1}w^{a^{-1}}a$, we get $w^{-1}w^{a^{-1}}\in W\cap A=1$ and hence $a$ centralizes $w$. Therefore $w$ is centralized by every element of $S$. Since $W$ is an irreducible $\mathbb{F}_pS$-module, it follows that $w=1$ and hence $\nor W A=1$. Since $G=WA$, it follows that $\nor G A =A$.

Finally, let $K$ be the kernel of the action of $G$ on $\Omega$. Then $K\leq G_1$ and, since $W$ is an irreducible $S$-module and since $W\nleq G_1$, we have $W\cap K=1$. As $G_1\cap Q=1$, we also have $Q\cap K=1$. This gives $K=1$ because from Maschke's theorem every irreducible $\mathbb{F}_pS$-submodule of $P$ is contained in either $Q$ or $W$. This shows that $G$ acts faithfully on $\Omega$.
}
\end{example}

Loosely speaking, Theorem~\ref{thm:main} shows that the groups in Example~\ref{ex1} are the building blocks of every permutation group having a self-normalizing abelian regular subgroup.

\begin{theorem}\label{thm:main}
Let $G$ be a permutation group with a maximal abelian regular subgroup $A$ such that $\nor G A = A$. Let $G_1$ be the stabilizer of the point $1$, let $N$ be the core of $A$ in $G$. Then there exist a prime $p$ and $Q$ and $S$ with $Q\neq 1\neq S$ such that
\begin{enumerate}
\item $A/N$ is cyclic of order coprime to $p$, \label{youb}
\item $G_1$ is an elementary abelian $p$-group, \label{yo}
\item $G/N\cong G_1N/N \rtimes A/N$ acts faithfully as an affine primitive group on the cosets of $A$ in $G$, \label{yam}
\item $N=\Z G=Q\times \cent S {G_1}$, \label{yi}
\item $G=(G_1\times Q)\rtimes S$,  \label{ya}
\item $A=Q\times S$, \label{ye}
\item $G_1\times Q$ is the unique Sylow $p$-subgroup of $G$, \label{yu}
\item  $\nor G {G_1}=\cent G {G_1}=G_1\times N$, \label{yog}
\item for all $s,s'\in G\setminus \nor G {G_1}$, we have $G_1G_s=G_1G_{s'}$. \label{youm}
\end{enumerate}

\end{theorem}

\begin{proof}
Write $\overline{G}=G/N$. (We adopt the ``bar'' convention and denote the group $XN/N$ by $\overline{X}$.) Note that since $A$ is not normal in $G$, we have $N<A$.

The group $\overline{G}$ acts faithfully as a primitive group on the cosets of $A$ in $G$ and the stabilizer $\overline{A}$ of the coset $A$ is abelian. Since $\overline{G}$ is primitive, we have that either $\Z {\overline{G}}=1$ or $|\overline{G}|$ is prime. The latter case contradicts the fact that $N<A<G$, and hence $\Z {\overline{G}}=1$. By Lemma~\ref{lemma:Fr}, there exists a prime $p$ such that $\soc(\overline{G})$ is an elementary abelian $p$-group and $\overline{A}$ is cyclic of order coprime to $p$. (This shows~(\ref{youb}).) Let $P$ be a Sylow $p$-subgroup of $G$ and note that $\overline{P}=\soc(\overline{G})$.

Note that $G=AG_1$ and that $A\cap G_1=1$. It follows that $N\cap G_1=1$ and hence $\overline{G_1}\cong G_1$ and $|\overline{G_1}|=|G_1|=|G:A|=|\overline{G}:\overline{A}|=|\overline{P}|$. Since $\overline{P}$ is the unique Sylow $p$-subgroup of $\overline{G}$, it follows that $\overline{G_1}=\overline{P}$ and $G_1$ is an elementary abelian $p$-group. (This shows~(\ref{yo}) and~(\ref{yam}).)

Let $g\in G\setminus A$. As $A$ is maximal in $G$ and $A=\nor G A$, we have that $G=\langle A ,A^g\rangle$. Since $N\leq A$ and $N\leq A^g$, we see that $A$ and $A^g$ centralize $N$ and hence $N\leq \Z G$. Since $\Z{\overline{G}}=1$ it follows that $N=\Z G$. (This shows the first equality in~\eqref{yi}.) Since $\overline{G}$ and $N$ are solvable, so is $G$.

Let $r$ be a prime divisor of $|G|$ different from $p$ and let $R$ be a Sylow $r$-subgroup of $G$ contained in $A$. If $\overline{R}\neq 1$ then, since $\overline{R}$ acts faithfully as a group of automorphisms on $\overline{G_1}$ and since $\overline{R}\cap\overline{G_1}=1$, we obtain $\nor {\overline{G}}{\overline{R}}\leq\overline{A}$. Since $\overline{A}$ is abelian, it follows that $\nor G R=A$ and hence $\nor G R=\cent G R$. From Burnside's normal $p$-complement theorem~\cite[Theorem~$5.13$]{Isaacs}, we see that $G=X\rtimes R$ for some Hall $r'$-subgroup $X$ of $G$.  If $\overline{R}=1$, then $R\leq N=\Z G$ and $R$ is central in $G$, and hence $G=X\times R$ for some Hall $r'$-subgroup $X$ of $G$.

Repeating the  argument in the previous paragraph for each prime divisor $r$ of $|G|$ different from $p$, we see that $G=P\rtimes S$, where $S$ is a Hall $p'$-subgroup of $G$. In particular, $P\unlhd G$. Moreover, as the Hall $p'$-subgroups are conjugate, we may choose the complement $S$ of $P$ in $G$ with $S\leq A$.

Let $Q=P\cap N$. Observe that $G_1\leq P$ because $P$ is a normal Sylow $p$-subgroup and $G_1$ is a $p$-group.   Since $p$ is coprime to $|\overline{A}|$ and $N\norml G$, we see that $P\cap A=Q$. Therefore,
$$P=P\cap G=P\cap G_1A=G_1(P\cap A)=G_1Q=G_1\times Q$$
where the last equality follows because $N=\Z G$. (This shows~(\ref{ya}),~(\ref{ye}) and~(\ref{yu}).) Note that this implies that $Q\not = 1$ as otherwise $G_1 = P\norml G$, which is not the case.  In particular, this shows that $P$ is abelian. Finally, note that $\cent A {G_1}=\Z G=N$ and hence $\cent S {G_1}=S\cap N$. Therefore,
$$N=A\cap N=(Q\times S)\cap N=Q\times (S\cap N)=Q\times \cent S {G_1}.$$
(This shows the second equality in~(\ref{yi}).)

Clearly, $\cent G {G_1}=G_1\times N$. We now show that $\nor G {G_1}=\cent G {G_1}$. Let $T=\nor G {G_1}$.  Since $P\leq \cent G {G_1}$, we see that $G/{\cent G{G_1}}$ is abelian and hence $T$ is normal in $G$. Now, $[T,P]=[T,G_1\times Q]=[T,G_1]$ since $Q\leq \Z G$. Moreover, $[T,P]$ is normal in $G$ because both $T$ and $P$ are. Since
$$[T,P]=[T,G_1]=[\nor G{G_1},G_1]\leq G_1$$
 and $G_1$ is core-free in $G$, we get $[T,G_1]=1$ and $T$ centralizes $G_1$, that is, $\nor G {G_1}=\cent G {G_1}$. (This shows~(\ref{yog}).) It follows that

$$\frac{G}{T} =  \frac{(P\rtimes S)}{P\times \cent S {G_1}} \cong \frac{S}{\cent S {G_1}}= \frac{S}{S\cap N}\cong \overline{S}\leq\overline{A}.$$

Recall that $\overline{A}$ is cyclic and hence so is $G/T$. Let $aT$ be a generator of $G/T$. Recall that $P=Q\times G_1$ and $Q\leq\Z G$; hence $[P,a]=[G_1,a]$ and $\cent P {a}=Q\times \cent {G_1} {a}$. Since $a$ acts irreducibly on $P/Q\cong\overline{G_1}\cong G_1$, it follows that $\cent {G_1} {a}=1$ and hence $\cent P {a}=Q$. Since $|aT|$ is coprime to $p$, we obtain from the coprime group action~\cite[8.4.2]{KurStell} that $P=[P,a]\times \cent P {a}=[G_1,a]\times Q$.

Similarly, for every $b\in \langle a\rangle$, we have $P=[P,b]\times \cent P {b}=[G_1,b]\times Q\times \cent {G_1} {b}$. Now, suppose $\cent {G_1} {b}>1$. Since $\langle b\rangle\norml\langle a\rangle$ and $a$ acts irreducibly on $\overline{G_1}$, we must have $\cent {G_1} {b}=G_1$ and $b\in\cent {G} {G_1}=T$.

We conclude that for every $b\in \langle a\rangle\setminus T$, we have $P=[G_1,b]\times Q$. Since $b$ is a power of $a$, we have $[G_1,b]\leq [G_1,a]$ and hence $[G_1,b]=[G_1,a]$. It follows that for every $s\in G\setminus T$, we have $[G_1,s]=[G_1,a]$ and hence
$$G_1G_s=G_1G_1^s=G_1[G_1,s]=G_1[G_1,a]=G_1G_{a}.$$
(This shows~(\ref{youm}).)
\end{proof}

Theorem~\ref{thm:main} is sufficient for the enumeration of Cayley digraphs on abelian groups. The corresponding result to enumerate Cayley graphs on abelian groups is Theorem~\ref{thm:main2}. Part of the hypothesis in the statement of Theorem~\ref{thm:main2} is somewhat technical, but this yields a conclusion that is easy to use and strong enough for our applications.

\begin{theorem}\label{thm:main2}
Let $G$ be a permutation group with an abelian regular subgroup $A$. Suppose that $\nor G A $ is generalised dihedral on $A$ and that $\nor G A $ is the unique group with the property that $A<\nor G A<G$. Then $\Z G$ is an elementary abelian $2$-group contained in $A$ and  $G=U\times \Z G$ where $G_1\leq U\cong\PGL(2,q)$ for some prime power $q\geq 3$, $A/ \Z G\cong\C_{q+1}$ and $U$ acts $2$-transitively on $U/G_1$. In particular, $G$ is endowed with the natural product action on $U/G_1\times \Z G$.
\end{theorem}

\begin{proof}
Let $B=\nor G A $. Since $A$ is a transitive abelian group, it follows that it is self-centralizing and hence $\Z G\leq A$. In particular, since $B$ does not centralize  $A$, $A$ is not an elementary abelian $2$-group. Let $\iota\in B\setminus A$. Then $\iota$ acts by inversion on $\Z G$ and hence $\Z G$ is an elementary abelian $2$-group.

As $B$ is maximal in $G$ and $B=\nor G A $, for $g\in G\setminus B$, we have $A< \langle A,A^g\rangle$ and hence either $\langle A,A^g\rangle=B$ or $\langle A,A^g\rangle=G$.

Suppose $\langle A,A^g\rangle=B$ for some $g\in G\setminus B$.  As $|B:A^g|=2$, we have $A^g\norml B$ and hence $B\leq \nor G {A^g}=(\nor G A)^g=B^g$, which gives $B=B^g$ and $g\in \nor G{B}$. Since $g\notin B$ and $B$ is maximal in $G$, it follows that $B\norml G$. Let $K$ be the group generated by elements of $B$ of order different from $2$. Clearly, $K$ is characteristic in $B$ and hence normal in $G$. Since all the elements in $B\setminus A$ have order $2$, $K\leq A$ and hence $K\leq A^g$. Let $x\in A^g\setminus A$. Since $A$ is not an elementary abelian $2$-group, $K\neq 1$ and there is an element $k\in K$ such that $k^2\neq 1$. Since $A^g\leq B$, we have $x\in B\setminus A$ and hence $x$ does not commute with $k$. This contradicts the fact that $A^g$ is abelian.

We may thus assume that $\langle A,A^g\rangle=G$, for every $g\in G\setminus B$. It follows that  $A\cap A^g\leq \Z G$, for $g\in G\setminus B$. Recall that $\Z G\leq A$ and hence $\Z G=A\cap A^g$ for every $g\in G\setminus B$.

Let $N$ be the core of $B$ in $G$.  Let $\overline{G}=G/N$. (Again, we adopt the ``bar'' convention and denote the group $XN/N$ by $\overline{X}$.) The action of $\overline{G}$ on the right cosets of $B$ in $G$ is faithful and, since $B$ is maximal in $G$, it is also primitive with point-stabilizer $\overline{B}$. It follows that either $\Z {\overline{G}}=1$ or $|\overline{G}|$ is prime. In the latter case, $B=N$ is normal in $G$. For $g\in G\setminus B$, we have $G=\langle A,A^g\rangle\leq B$, which is a contradiction. Thus $\Z {\overline{G}}=1$ and hence $\Z G\leq N$. We will now prove the following.

\smallskip

\noindent\textsc{Claim. } $G=U\times \Z G$ where $U\cong\PGL(2,q)$ for some prime power $q\geq 3$, $A/\Z G\cong\C_{q+1}$, and $G/\Z G$ is $2$-transitive on $U/G_1$.

\smallskip

\noindent First we consider the case when $N\nleq A$. It follows that $B=NA = AN$ and $\overline{B}\cong A/(A\cap N)$ is abelian. From Lemma~\ref{lemma:Fr}, it follows that $\overline{G}=\overline{T}\rtimes \overline{B}$ for some $T$ with $N\leq T$, with $\overline{T}$ an elementary abelian $p$-group and $\overline{B}$ cyclic of order coprime to $p$. In particular, $N=T\cap B$.

Fix $g\in G\setminus B$. Since $B=NA$ and $|B:A|=2$, we see that $|N:(A\cap N)|=2$ and $|N:(A^g\cap N)|=2$. Since $A\cap A^g=\Z G\leq N$ it follows $(A\cap N)\cap(A^g\cap N) = \Z G$ and hence $|N:\Z G|=2$ or $4$. In particular, $N$ is a $2$-group. Let $n\in N\setminus A$. Since $B=NA$, we see that $n$ acts by inversion on $A$. In particular, for every $x\in A$, we obtain that $x^{-2}n=x^{-1}(nxn^{-1})n=x^{-1}nx\in N$ and hence $x^2\in N$. Since $\overline{B}\cong A/(A\cap N)$ is cyclic and since $N$ contains the square of each element of $A$, we obtain $|A:(A\cap N)|=2=|\overline{B}|$. Since $\overline{G}$ is primitive with point-stabilizers of order $|\overline{B}|=2$, it follows that it is dihedral of order $2p$ and $|\overline{T}|=p$ for some odd prime $p$. As $|\overline{B}|=2$ and $N$ is a $2$-group, we obtain that $B$ is a Sylow $2$-subgroup of $G$ and $|G|=p|B|$.  Let $Q$ be a Sylow $p$-subgroup of $T$.

Suppose that $N/\Z G$ is central in $T/\Z G$. Since $T=QN$ and $p>2$, we have $$\frac{T}{\Z G}=\frac{Q\Z G}{\Z G}\times \frac{N}{\Z G}.$$ In particular, since $p>2$,  the group $Q\Z G/\Z G$ is characteristic in $T/\Z G$ and hence normal in $G/\Z G$. Thus $Q\Z G\norml G$. Let $R=QA$. This is a subgroup of $G$ because $Q\Z G$ is normal in $G$ and $\Z G\leq A$. Since $Q$ is a $p$-group and $B$ is a $2$-group, we get $Q\cap B=1$. As $|B:A|=2$ and $G=RB$, it follows that $|G:R|=2$. We have shown that $R$ is a subgroup of $G$ containing $A$ which is neither $A$, $B$ or $G$. This is a contradiction.

Therefore $N/\Z G$ is not central in $T/\Z G$.  Recall that $N/\Z G$ is a normal Sylow $2$-subgroup of $T/\Z G$ of order at most $4$. It follows that $N/\Z G\cong\C_2\times\C_2$, $p=3$ and $T/\Z G\cong\Alt(4)$. Since $B$ is generalized dihedral and $A\cap N < N \leq B$, there exists an involution $x\in N\setminus A$. Let $t$ be an element of $T$ of order $3$. The action of $t$ on the non-identity elements of $N/\Z G$ is transitive hence every coset of $\Z G$ in $N$ contains an involution. It follows that $N$ is elementary abelian and splits over $\Z G$. Since $|T:N|=3$, $T$ splits over $N$ and hence also over $\Z G$. Similarly, since $B$ is generalized dihedral and $B\not\leq T$, there is an involution in $G\setminus T$. In particular, $G$ splits over $T$ and hence also over $\Z G$.

It follows that $G=U\times \Z G$ for some $U\cong \Alt(4)\rtimes \C_2$. Since $\overline{G}$ is not abelian, we conclude that $U\not\cong  \Alt(4)\times C_2$ and hence $U\cong\Sym(4)\cong\PGL(2,3)$. Since $B\cap U$ is a Sylow $2$-subgroup of $U$, it is isomorphic to $\D_4$ and hence $A/\Z G\cong\C_4$. This concludes the proof of our claim in the case when $N\nleq A$.

We now assume that $N\leq A$. Let $g\in G\setminus B$. Then $N\leq A\cap A^g=\Z G$ and hence $N=\Z G$. Let $\overline{T}$ be the socle of $\overline{G}$. (Here $T$ is a subgroup of $G$ with $N\leq T$.)

Suppose that $\overline{T}$ is elementary abelian. It follows that $\overline{G}=\overline{T}\rtimes\overline{B}$ and hence $T\cap B= N$. Let $R=AT$ and note that $|G:R|=2$ because $|B:A|=2$. Moreover, $G=BR$.  We have shown that $R$ is a subgroup of $G$ containing $A$ which is neither $A$, $B$ or $G$. This is a contradiction.

We may thus assume that $\overline{T}$ is not elementary abelian. Note that $G=G_1A$ and $G_1\cap A=1$. It follows that $\overline{G}=\overline{G_1}\,\overline{A}$ and $\overline{G_1}\cap \overline{A}=1$ (for the last equality use $N\leq A$). By applying Proposition~\ref{prop:ch} to $\overline{G}$ with $L=\overline{G_1}$, we see that $\overline{T}\cong\PSL(2,q)$ for some prime power $q\geq 4$, that $\overline{G}\cong\PGL(2,q)$, that $\overline{A}\cong C_{q+1}$, and that $\overline{G}$ is $2$-transitive. It remains to show that $G$ splits over $\Z G$.

Let $H$ be the last term of the derived series of $G$. Since $T/\Z G\cong\PSL(2,q)$ is perfect, it follows that $T=H\Z G$ and hence $\overline{H}\cong H/(H\cap\Z G)\cong\PSL(2,q)$ therefore $H\cap\Z G=\Z H$. In particular, $\Z H\leq H=H'$ and hence $H$ is a quotient of the universal central extension of $\PSL(2,q)$.

Suppose that $H\cong \overline{H}$. Then $H\cap\Z G=1$ and hence $T=H\times \Z G$. In particular, $T$ splits over $\Z G$. Since $B$ is generalized dihedral and $B\not\leq T$, there is an involution in $G\setminus T$. It follows that $G$ splits over $T$ and hence also over $\Z G$. Thus $G=U\times\Z G$ for some $U\cong\PGL(2,q)$ and the claim follows.

Suppose now that $H\not\cong\overline{H}$. Recall that the Sylow $2$-subgroup of the Schur multiplier of $\PSL(2,q)$ has order $2$ (see~\cite[page~xvi, Table~5]{ATLAS}.). It follows that  $H\cong\SL(2,q)$ and $T=H\times V$ for some subgroup $V$ of index $2$ in $\Z G$.  In particular, every involution of $T$ is central in $G$. Since $|G:T|=2$ and $B\not\leq T$, we have $|B:T\cap B|=2$. Moreover, since $|B:A|=2$ and $A\not\leq T$ it follows that $|T\cap B:T\cap A|=2$. In particular, there is an involution in $T\cap B$ which acts by inversion on $A$. This contradicts the fact that every involution in $T$ is central in $G$. ~$_\blacksquare$

\smallskip

We now show that, by replacing $U$ with a subgroup of $G$ isomorphic to $U$ if necessary, we have $G_1\leq U$. Suppose that $G_1\nleq U$. Clearly $G_1\cong \overline{G_1}\cong\FF_q\rtimes\C_{q-1}$. Let $G_1^2=\langle g^2\mid g\in G_1\rangle$. An easy computation yields that $G_1^2=G_1$ if $q$ is even and  $G_1^2\cong\FF_q\rtimes\C_{(q-1)/2}$ if $q$ is odd. Let $g\in G_1$. Then $g=uz$ for some $u\in U$ and some $z\in \Z G$. Thus $g^2=(uz)^2=u^2z^2=u^2\in U$ and hence $G_1^2\leq U$. Since $G_1\nleq U$, it follows that $q$ is odd and $G_1\cap U=G_1^2\cong\FF_q\rtimes\C_{(q-1)/2}$. Since $[U,U]\cong \PSL(2,q)$, it can be seen that $G_1^2\leq [U,U]$ and hence $G_1^2=G_1\cap [U,U]$. Let $g\in G_1\setminus U$. Since $|G_1:G_1\cap [U,U]|=2$, it follows that $|\langle g\rangle[U,U]|=2|[U,U]|=|U|$. Write $g=uz$ for some $u\in U$ and some $z\in \Z G$. Note that $G_1\leq \langle g\rangle [U,U]$, and that $g$ acts on $[U,U]$ as $u$, hence $\langle g\rangle[U,U]\cong U$ and we may replace $U$ by $\langle g\rangle[U,U]$.

Since $G=U\times \Z G$ and $G_1=G_1\times 1$, we see that $G$ is endowed with the natural product action on $U/G_1\times \Z G$, which concludes the proof.
\end{proof}

\section{An application to Cayley digraphs on abelian groups}
\label{sec:AppDigraph}

\begin{definition}\label{def:semiwreath}
Let $A$ be an abelian group and let $1< H\leq K<A$. We say that the Cayley digraph $\Cay(A,S)$ is a \emph{generalized wreath} digraph with respect to $(H,K,A)$ if $S\setminus K$ is a union of $H$-cosets.
\end{definition}

Definition \ref{def:semiwreath} is fairly natural and generalizes the well-established definition of wreath digraphs (which is the case $H=K$). Intuitively, in the digraph $\Cay(A,S)$, for $v,w\notin K$, if we have an arc from $v$ to $w$ with $vK$ and $wK$ two distinct $K$ cosets, then there is also an arc from $v$ to $wh$, for every $h\in H$. We now give an application of Theorem~\ref{thm:main} to the study of Cayley digraphs on abelian groups.

\begin{theorem}\label{thm:graphmain}
Let $G$ be a permutation group on $\Omega$ with a proper self-normalizing abelian regular subgroup $A$. Then $|A|$ is not a prime power and there exist two groups $H$ and $K$ with $1<H\leq K<A$, and for every digraph $\vGa$ with $G\leq \Aut(\vGa)$, we have that $\vGa$ is a generalized wreath digraph with respect to $(H,K,A)$.
\end{theorem}
\begin{proof}
Let $M$ be a subgroup of $G$ with $A$ maximal in $M$. Clearly $\nor M A=A<M$ and hence, by replacing $G$ by $M$, we may assume that $A$ is maximal in $G$.  This allows us to apply Theorem~\ref{thm:main} and we adopt the notation from its statement. We see immediately that $|A|$ is not a prime power.

Let $T=\nor G{G_1}$. By Theorem~\ref{thm:main}~(\ref{yi}), (\ref{yu}) and~(\ref{yog}), we have that $T$ contains the unique Sylow $p$-subgroup of $G$ and hence $G_y\leq T$ for every $y\in\Omega$. Since $G_1$ is normal in $T$, it follows that $G_1G_y$ is a subgroup of $T$ and $G_1G_y=G_yG_1$. Let $s\in G\setminus T$ and let $H=G_1G_s\cap A$. By Theorem~\ref{thm:main}~(\ref{youm}), $H$ does not depend on the choice of $s$. If $H=1$ then, by order considerations,  $G_1G_s=G_1$ and hence $s\in T$, which is a contradiction. Therefore $H\neq 1$.

Let $K=N$. By Theorem~\ref{thm:main}~(\ref{yog}), $T\cap A = (G_1\times N)\cap A=(G_1\cap A)\times N = N= K$ and hence $H\leq K<A$. Since $A$ is a regular subgroup of $G$, we can identify $\Omega$ with $A$. Let $x$ in $\Omega\setminus K$. Since $T\cap A=K$, we have $x\notin T$ and $H=G_1G_x\cap A$. Since $G_1G_x$ is a subgroup containing $G_1$, it follows that $x^{G_1G_x}$ is a block of imprimitivity for $G$ and hence also for $A$. Moreover, $G_1G_x$ is the stabilizer of this block in $G$, hence $H=G_1G_x\cap A$ is the stabilizer of this block in $A$, therefore $x^{G_1G_x}$ is an $H$-coset. On the other hand, $x^{G_1}=x^{G_xG_1}=x^{G_1G_x}$. We have shown that every $G_1$-orbit on $\Omega\setminus K$ is an $H$-coset.  If follows that every digraph $\vGa$ with $G\leq \Aut(\vGa)$ is a generalized wreath digraph with respect to $(H,K,A)$.\end{proof}

Moving from Cayley digraphs to Cayley graphs, the theorem corresponding to Theorem~\ref{thm:graphmain} is Theorem~\ref{thm:graphmain2}, but we first need the following definition. Given two graphs $\vGa_1=(\mathcal{V}_1,\mathcal{A}_1)$ and $\vGa_2=(\mathcal{V}_2,\mathcal{A}_2)$, the \emph{direct product} $\vGa_1\times \vGa_2$ of $\vGa_1$ and $\vGa_2$ is the graph with vertex-set $\mathcal{V}_1\times\mathcal{V}_2$ and all arcs of the form $((u_1,u_2),(v_1,v_2))$ where $(u_1,v_1)\in\mathcal{A}_1$ and $(u_2,v_2)\in\mathcal{A}_2$.

\begin{theorem}\label{thm:graphmain2}
Let $G$ be a permutation group with an abelian regular subgroup $A$. Suppose that $\nor G A$ is a proper subgroup of $G$ and is generalized dihedral on $A$. Then one of the following occurs:
\begin{enumerate}
\item $|A|$ is not a prime power and there exist two groups $H$ and $K$ with $1<H\leq K<A$, and for every graph $\Gamma$ with $G\leq \Aut(\Gamma)$, we have that $\Gamma$ is a generalized wreath graph with respect to $(H,K,A)$; or
\item there exist two groups $C$ and $Z$ with $A=C\times Z$, with $C\cong \C_t$ for some $t\geq 4$ and with $Z$ an elementary abelian $2$-group,  such that, for every graph $\Gamma$ with $G\leq \Aut(\Gamma)$, we have that $\Gamma$ is isomorphic to the direct product of $\Lambda$  with a Cayley graph over $Z$, where $\Lambda$ is either complete or edgeless, possibly with a loop at each vertex.
\end{enumerate}
\end{theorem}
\begin{proof}
Let $\nor G A = B$ and let $M$ be a subgroup of $G$ with $B$ maximal in $M$. Clearly $\nor M A=B<M$ and hence, by replacing $G$ by $M$, we may assume that $B$ is maximal in $G$.  Now, suppose that there exists a group $X$ with $A<X<G$, and $X\neq B$. Since $\nor G A = B$ and $A$ is maximal in $B$, it follows that $\nor X A=A$. We may then apply Theorem~\ref{thm:graphmain} to conclude that part~(1) holds.

We may thus assume that the only proper subgroups of $G$ containing $A$ are $A$ and $B$ and hence the hypothesis of Theorem~\ref{thm:main2} is satisfied. It then follows that $\Z G$ is an elementary abelian $2$-group contained in $A$, that  $G=U\times \Z G$ where $G_1\leq U\cong\PGL(2,q)$ for some prime power $q\geq 3$, that $A/ \Z G\cong\C_{q+1}$, that $U$ acts $2$-transitively on $U/G_1$ and that $G$ is endowed with the natural product action on $U/G_1\times \Z G$.

As $G$ is endowed with the canonical product action, we have $A=C\times\Z G$ for some $C\leq U$ with $C\cong\C_{q+1}$. Now $G=U\times \Z G$ acts by product action  on $C\times\Z G$.

Let $\Gamma$ be a graph with $G\leq \Aut(\Gamma)$. In particular, $\Gamma=\Cay(A,S)$ for some subset $S$ of $A$. As  $U$ is $2$-transitive in its action on the cosets of $G_1$, we have $S=S'\times S''$, where $S'\in\{\emptyset,\{1_C\},C\setminus\{1_C\},C\}$ and $S''$ is a subset of $\Z G$.  From this description of $S$ it follows that $\Gamma$ is the direct product of $\Cay(C,S')$ and $\Cay(\Z G,S'')$. The proof then follows by taking $Z=\Z G$ and $t=q+1$.
\end{proof}

\section{Enumeration}
\label{sec:EnumerationDigraphs}

If $G$ is a group of order $n\geq 2$, then it is at most $\lfloor\log_2(n)\rfloor$-generated and hence $|\Aut(G)|\leq n^{\log_2(n)}= 2^{(\log_2(n))^2}$. Similarly, any subgroup of $G$ is also at most $\lfloor\log_2(n)\rfloor$-generated and hence $G$ has at most $n^{\log_2(n)}= 2^{(\log_2(n))^2}$ distinct subgroups. These facts will be used repeatedly.

\subsection{Enumeration of Cayley digraphs on abelian groups}\label{secb:EnumerationDigraphs}
We first deal with the enumeration of digraphs because it is easier than the enumeration of graphs. Moreover, the general outline of the proof is the same,  hence this section serves as a template for the next one. Our first goal is to prove two technical lemmas which, loosely speaking, give an upper bound on the number of ``bad'' subsets, in view of Theorem~\ref{thm:graphmain}.

\begin{lemma}\label{lemma:semiwreathdigraphcounting}
Let $A$ be a group of order $n$. The number of subsets $S$ of $A$ such that there exist two groups $H$ and $K$ with $1<H\leq K<A$ and such that $S\setminus K$ is a union of left (or right) $H$-cosets is at most $2^{3n/4+2(\log_2(n))^2}$.
\end{lemma}
\begin{proof}
As noted earlier, $A$ has at most $2^{(\log_2(n))^2}$ distinct subgroups hence there are at most $2^{2(\log_2(n))^2}$ ways of choosing $H$ and $K$. We now count the number of possibilities for $S$ for fixed $H$ and $K$. Let $h=|H|$ and let $k=|K|$. Then $A$ admits exactly $2^{k+\frac{n-k}{h}}$ subsets satisfying the hypothesis. Since $h\geq 2$ and $k\leq n/2$, we have $k+\frac{n-k}{h}\leq 3n/4$ and the result follows.
\end{proof}

Lemma~\ref{lemma:normalizeddigraphcounting} is a weaker version of a result from~\cite{Ba2}, but the proof is very easy.

\begin{lemma}\label{lemma:normalizeddigraphcounting}
Let $G$ be a group of order $n$. The number of subsets of $G$ which are normalized by some element of $\mathrm{Aut}(G)\setminus\{1\}$ is at most $2^{3n/4+(\log_2(n))^2}$.
\end{lemma}
\begin{proof}
Recall that $|\mathrm{Aut}(G)|\leq 2^{(\log_2(n))^2}$. We now count the number of subsets which are normalized by a fixed $\varphi\in\mathrm{Aut}(G)\setminus\{1\}$. Note that $\varphi$ induces orbits of length $1$ on $\cent G\varphi$ and of length at least $2$ on $G\setminus\cent G\varphi$. Let $c=|\cent G\varphi|$. The number of subsets of $G$ which are normalized by $\varphi$ is at most $2^{c+(n-c)/2}=2^{n/2+c/2}$. Since $c\leq n/2$, we have ${n/2+c/2}\leq 3n/4$ and the result follows.
\end{proof}

Theorem~\ref{thm:graphmain} is combined with Lemmas~\ref{lemma:semiwreathdigraphcounting} and~\ref{lemma:normalizeddigraphcounting} to prove Theorem~\ref{th:epsilon1}. Before proceeding, we set some notation which will be used in this section and the next.

Let $2^A$ denote the set of subsets of $A$, let $2^A_{DRR}$ denote the set of subsets $S$ of $A$ such that $\Cay(A,S)$ is a DRR, let $2^A_{gw}$ denote the set of subsets $S$ of $A$ with the property that $\Cay(A,S)$ is a generalized wreath digraph with respect to $(H,K,A)$ for some $H,K\leq A$, and let $2^A_{nor}$ denote the set of subsets $S$ of $A$ with the property that $\Cay(A,S)$ admits an element of $\Aut(A)\setminus\{1\}$ as a digraph automorphism. Finally, let $2^A_{bad}=2^A_{gw} \cup 2^A_{nor}$ and let $2^A_{good}=2^A\setminus 2^A_{bad}$.

\begin{proof}[Proof of Theorems~$\ref{th:main1}$ and~$\ref{th:epsilon1}$]
It follows  from Theorem~\ref{thm:graphmain} that $2^A_{good}\subseteq 2^A_{DRR}$ and hence $2^A\setminus 2^A_{DRR}\subseteq 2^A_{bad}$. By Lemmas~\ref{lemma:semiwreathdigraphcounting} and~\ref{lemma:normalizeddigraphcounting}, we have   $|2^A_{gw}|\leq 2^{3n/4+2(\log_2(n))^2}$  and $|2^A_{nor}|\leq 2^{3n/4+(\log_2(n))^2}$ therefore $|2^A_{bad}|\leq 2^{3n/4+2(\log_2(n))^2 +1}$. This shows Theorem~\ref{th:epsilon1}. Since $|2^A|=2^n$, we have $|2^A_{bad}|/|2^A|\to 0$ as $n\to\infty$ and Theorem~\ref{th:main1} follows.
\end{proof}

\subsection{Enumeration of Cayley graphs on abelian groups}\label{secb:EnumerationGraphs}
The general outline of this section is the same as Section~\ref{secb:EnumerationDigraphs}'s. We first prove a few upper bounds on the number of ``bad'' subsets, this time with respect to Theorem~\ref{thm:graphmain2}.

\begin{lemma}\label{lemma:strangecounting}
Let $A$ be an abelian group of order $n$. The number of quadruples $(C,Z,S',S'')$ with $A=C\times Z$, $C$ a cyclic group of order $t\geq 4$, $Z$ an elementary abelian $2$-group, $S'\in \{C,\emptyset,\{1\},C\setminus\{1\}\}$, and $S''\subseteq Z$ is at most $2^{n/4+2\log(n)-1}$.
\end{lemma}
\begin{proof}
Clearly, we may assume that $A=\langle \lambda\rangle\times Z'$ for some elementary abelian $2$-group $Z'$ and some $\langle \lambda\rangle$ of order $t\geq 4$. If $t$ is odd, then this decomposition is unique. If $t$ is even, then the number of choices for $C$ is $|Z'|$ ($C=\langle \lambda k\rangle$ for some $k\in Z'$), while the number of choices for $Z$ is at most the number of subgroups of index $2$ in $\langle\lambda^{|\lambda|/2}\rangle\times Z'$, which is at most $2|Z'|$. Once $C$ and $Z$ are fixed we have $4$ choices for $S'$ and $2^{|Z|}$ choices for $S''$. Since $|Z|=|Z'|\leq n/4$, it follows that there are at most $|Z'|\cdot 2|Z'|\cdot 4\cdot 2^{|Z|}\leq n^22^{n/4-1}=2^{n/4+2\log(n)-1}$ quadruples.
\end{proof}

\begin{lemma}\label{lemma:semiwreathcounting}
Let $n$ be an integer that is not a power of $2$, let $A$ be an abelian group of order $n$ and let $m$ be the number of elements of order at most $2$ in $A$. Then the number of inverse-closed subsets $S$ of $A$ such that there exist two groups $H$ and $K$ with $1<H\leq K<A$, and such that $S\setminus K$ is a union of $H$-cosets is at most $2^{m/2+11n/24+2(\log_2(n))^2}$.
\end{lemma}
\begin{proof}
As before, there are at most $2^{2(\log_2(n))^2}$ ways of choosing $H$ and $K$. We now count the number of possibilities for $S$ for fixed $H$ and $K$.

Let $h=|H|$, let $k=|K|$, let $j$ be the number of elements of order at most $2$ in $K$ and let $i$ be the number of elements of $A\setminus K$ whose square lies in $H$. Note that $x^2\in H$ if and only if $xH=(xH)^{-1}$ and hence $A$ admits exactly $2^{j+\frac{k-j}{2}+\frac{i}{h}+\frac{n-k-i}{2h}}$ inverse-closed subsets $S$ such that $S\setminus K$ is a union of $H$-cosets. Note that $j\leq m$ and $\frac{k}{2}+\frac{i}{h}+\frac{n-k-i}{2h}=\frac{n}{2h}+\frac{i}{2h}+k\left(\frac{h-1}{2h}\right)$,  hence it suffices to show that $\frac{n}{2h}+\frac{i}{2h}+k\left(\frac{h-1}{2h}\right)\leq 11n/24$.

Note that $i\leq n-k$ and $k\leq n/2$ hence $\frac{n}{2h}+\frac{i}{2h}+k\left(\frac{h-1}{2h}\right)\leq \frac{n}{h}+k\left(\frac{h-2}{2h}\right)\leq n\left(\frac{h+2}{4h}\right)$. This concludes the proof when $h\geq 3$.

If $h=2$, then an element whose square lies in $H$ must be contained in the Sylow $2$-subgroup of $A$. Since $A$ is not a $2$-group, there are at most $n/3$ such elements and hence $i\leq n/3$. Since $k\leq n/2$, it follows that $\frac{n}{2h}+\frac{i}{2h}+k\left(\frac{h-1}{2h}\right)\leq n/4+n/12+n/8=11n/24$. This concludes the proof.
\end{proof}

\begin{lemma}\label{lemma:normalizedcounting}
Let $A$ be an abelian group of order $n$ and of exponent greater than $2$, let $m$ be the number of elements of order at most $2$ in $A$ and let $\iota:A\to A$ be the automorphism defined by $\iota:x\mapsto x^{-1}$. Then the number of inverse-closed subsets of $A$ which are normalized by some element of $\mathrm{Aut}(A)\setminus\{1,\iota\}$ is at most $2^{m/2+11n/24+(\log_2(n))^2}$.
\end{lemma}
\begin{proof}
Recall that $|\mathrm{Aut}(A)|\leq 2^{(\log_2(n))^2}$. Let $\varphi\in\mathrm{Aut}(A)\setminus\{1,\iota\}$. Note that an inverse-closed subset is normalized by $\varphi$ if and only if it is normalized by $\langle\iota,\varphi\rangle$. It thus suffices to show that the number of inverse-closed subsets of $A$ which are normalized by $\langle\iota,\varphi\rangle$ is at most $2^{m/2+11n/24}$.

Note that $|\iota|=2$, that $\iota$ commutes with every automorphism of $A$ and that $m=|\cent A\iota|$. Let $c=|\cent A\varphi|$ and let $k=|\cent A{\iota,\varphi}|$.

Suppose first that $|\varphi|$ is divisible by some odd prime $p$. Replacing $\varphi$ by a suitable power, we may assume without loss of generality that $|\varphi|=p$. Observe that $\langle\iota,\varphi\rangle=\langle\iota \varphi\rangle$ is cyclic of  order $2p$. Now, $\iota\varphi$ induces orbits of length $1$ on $\cent A{\iota,\varphi}$, of length $2$ on $\cent A\varphi\setminus \cent A\iota$, of length $p$ on $\cent A \iota\setminus \cent A \varphi$, and of length $2p$ on $A\setminus (\cent A\iota\cup \cent A \varphi)$. It follows  that the number of subsets of $A$ which are normalized by $\langle\iota,\varphi\rangle$ is
$$2^{k}2^{(c-k)/2}2^{(m-k)/p}2^{(n-(c+m-k))/(2p)}\leq 2^{k/3+c/3+m/6+n/6}\leq 2^{m/2+n/3},$$
where the first inequality follows from the fact that $p\geq 3$ and the last inequality from $k\leq m$ and $c\leq n/2$.

Suppose now that $|\varphi|$ is a power of $2$. We first assume that $\iota\in \langle\varphi\rangle$ and observe that $\cent A \varphi\leq \cent A \iota$. By replacing $\varphi$ by a suitable power, we may assume that $\varphi^2=\iota$ and hence $\varphi$ induces orbits of length $1$ on $\cent A \varphi$, of length $2$ on $\cent A \iota\setminus \cent A \varphi$, and of length $4$ on $A\setminus \cent A \iota$.  It follows  that the number of subsets of $A$ which are normalized by $\langle\varphi\rangle$ is
$$2^{c}2^{(m-c)/2}2^{(n-m)/4}= 2^{c/2+m/4+n/4}\leq 2^{m/2+3n/8},$$
where we have used the facts that $m\leq n/2$ and $c\leq m$.

It remains to consider the case $\iota\notin\langle\varphi\rangle$. Replacing $\varphi$ by a suitable power, we may assume that $|\varphi|=2$  and $\langle\iota,\varphi\rangle$ is an elementary abelian group of order $4$. It follows that $\langle\iota,\varphi\rangle$ induces orbits of length $1$ on $\cent A{\iota,\varphi}$, of length $2$ on $(\cent A \varphi\cup \cent A\iota\cup\cent A {\iota\varphi})\setminus\cent A{ \iota,\varphi}$, and of length $4$ on $A\setminus(\cent A \varphi\cup \cent A\iota\cup\cent A {\iota\varphi})$. Let $c'=|\cent  A{\iota \varphi}|$. The number of subsets of $A$ which are normalized by $\langle\iota,\varphi\rangle$ is
$$2^{k}2^{(m-k)/2}2^{(c-k)/2}2^{(c'-k)/2}2^{(n-(m+c+c'-2k))/4}=2^{m/4+c/4+c'/4+n/4}.$$

If one of $c$ or $c'$ is at most $n/3$, then $c/4+c'/4+n/4\leq n/8+n/12+n/4=11n/24$ and the conclusion holds. We may thus assume that $c=c'=n/2$. If $m=n/2$, then $m/4+c/4+c'/4+n/4=m/2+3n/8$ and the conclusion holds. We thus assume that $m<n/2$ and thus $\cent A \iota=\cent A{\iota\varphi,\varphi}$ has index $4$ in $A$. Thus $m=n/4$ and $m/4+c/4+c'/4+n/4=m/2+7n/16$.
\end{proof}

The upper bounds in Lemmas~\ref{lemma:semiwreathcounting} and~\ref{lemma:normalizedcounting} should not be taken too seriously since they are probably far from best possible, but they are sufficient to prove Theorem~\ref{th:epsilon2}.

We now introduce notation corresponding to that in the preceding section but for inverse-closed subsets. Let $2^A_*$ denote the set of inverse-closed subsets of $A$ and let $2^A_{*Small}$ denote the set of inverse-closed subsets $S$ of $A$ such that $\Aut(\Cay(A,S))=A\rtimes\langle\iota\rangle$. Let $2^A_{*ex}$ denote the set of inverse-closed subsets $S$ of $A$ with $A=C\times Z$ and $S=S'\times S''$, where $C$ is a cyclic group of order at least $4$, $Z$ is an elementary abelian $2$-group, $S'\in \{C,\emptyset,\{1\},C\setminus\{1\}\}$, and $S''\subseteq Z$, let $2^A_{*gw}$ denote the empty set if $|A|$ is a prime power and, otherwise, let $2^A_{*gw}$ denote the set of inverse-closed subsets $S$ of $A$ with the property that $\Cay(A,S)$ is a generalized wreath graph with respect to $(H,K,A)$, for some subgroups $H,K\leq A$.  Let $2^A_{*nor}$ denote the set of inverse-closed subsets $S$ of $A$ with the property that $\Cay(A,S)$ admits an element of $\Aut(A)\setminus\{1,\iota\}$ as a graph automorphism. Finally, let $2^A_{*bad}=2^A_{*ex} \cup 2^A_{*gw} \cup 2^A_{*nor}$ and let $2^A_{*good}=2^A_*\setminus 2^A_{*bad}$.

\begin{proof}[Proof of Theorems~$\ref{th:main2}$ and~$\ref{th:epsilon2}$]
If $A$ has exponent at most $2$, $2^A=2^A_*$, and every Cayley digraph on $A$ is actually a Cayley graph, and the result follows from Theorems~$\ref{th:main1}$ and~$\ref{th:epsilon1}$. We thus assume that $A$ has exponent greater than $2$. Let $\iota:A\to A$ be the automorphism defined by $\iota:x\mapsto x^{-1}$, let $B=A\rtimes\langle\iota\rangle$ and observe that $B$ is generalized dihedral over $A$. Let $m$ be the number of elements of order at most $2$ in $A$.

It follows  from Theorem~\ref{thm:graphmain2} that $2^A_{*good}\subseteq 2^A_{*Small}$ and hence $2_*^A\setminus 2^A_{*Small}\subseteq 2^A_{*bad}$. By Lemmas~\ref{lemma:strangecounting},~\ref{lemma:semiwreathcounting} and~\ref{lemma:normalizedcounting}, we have   $|2^A_{*ex}|\leq 2^{n/4+2\log(n)-1}$, $|2^A_{*gw}|\leq 2^{m/2+11n/24+2(\log_2(n))^2}$  and $|2^A_{*nor}|\leq 2^{m/2+11n/24+(\log_2(n))^2}$. It follows that $|2^A_{*bad}|\leq 2^{m/2+11n/24+2(\log_2(n))^2+2}$. This shows Theorem~\ref{th:epsilon2}. Since $|2_\ast^A|=2^m2^{(n-m)/2}=2^{m/2+n/2}$, we have $|2^A_{*bad}|/|2_*^A|\to 0$ as $n\to\infty$ and Theorem~\ref{th:main2} follows.
\end{proof}

\section{Unlabeled digraphs}\label{sec:unlabelled}
An  \emph{unlabeled} (di)graph is simply an equivalence class of (di)graphs under the relation ``being isomorphic to''. We will often identify a representative with its class. Using this terminology, we have the following unlabeled version of Theorems~\ref{th:main1} and~\ref{th:epsilon1}.

\begin{theorem}\label{th:unlabelledmain1}
Let $A$ be an abelian group of order $n$. Then the ratio of the number of unlabeled DRRs on $A$ over the number of unlabeled Cayley digraphs on $A$ tends to $1$ as $n\to\infty$.
\end{theorem}
\begin{proof}
Let $\UDRR(A)$ denote the set of unlabeled DRRs on $A$, let $S_1,S_2\in 2^A_{DRR}$ and let $\vGa_1=\Cay(A,S_1)$ and $\vGa_2=\Cay(A,S_2)$.
Suppose that $\vGa_1\cong\vGa_2$ and let $\varphi$ be a digraph isomorphism from $\vGa_1$ to $\vGa_2$. Note that $\varphi$ induces a group automorphism from $\Aut(\vGa_1)=A$ to $\Aut(\vGa_2)=A$. In particular, $\varphi\in \Aut(A)$ and $S_1$ and $S_2$ are conjugate via an element of $\Aut(A)$. This shows that $|\UDRR(A)|\geq |2^A_{DRR}|/|\Aut(A)|$. By Theorem~\ref{th:epsilon1}, we have $|2^A_{DRR}|\geq 2^n-2^{3n/4+2(\log_2(n))^2+1}$. Since $|\Aut(A)|\leq  2^{(\log_2(n))^2}$, it follows that
$$|\UDRR(A)|\geq 2^{n-(\log_2(n))^2}-2^{3n/4+(\log_2(n))^2+1}.$$
Let $\UCDN(A)$ denote the set of unlabeled Cayley digraphs on $A$ that are not DRRs. Note that
$$\frac{|\UDRR(A)|}{|\UDRR(A)|+|\UCDN(A)|}=1-\frac{|\UCDN(A)|}{|\UDRR(A)|+|\UCDN(A)|}\geq 1-\frac{|\UCDN(A)|}{|\UDRR(A)|}.$$
By Theorem~\ref{th:epsilon1}, we have $|\UCDN(A)|\leq 2^{3n/4+2(\log_2(n))^2+1}$ and thus
$$\frac{|\UCDN(A)|}{|\UDRR(A)|}\to 0,$$
as $n\to\infty$. This completes the proof.
\end{proof}

We now prove the corresponding theorem for unlabeled graphs.

\begin{theorem}\label{th:unlabelledmain2}
Let $A$ be an abelian group of order $n$ and let $B=A\rtimes\langle\iota\rangle$. Then the ratio of the number of unlabeled Cayley graphs on $A$ with automorphism group $B$ over the number of unlabeled Cayley graphs on $A$ tends to $1$ as $n\to\infty$.
\end{theorem}
\begin{proof}
Let $\UAGRR(A)$ denote the set of unlabeled Cayley graphs on $A$ with automorphism group $B$. If $A$ has exponent at most $2$, then $\iota=1$ and every Cayley digraph on $A$ is actually a Cayley graph, and the result follows from Theorem~\ref{th:unlabelledmain1}. We thus assume that $A$ has exponent greater than $2$. It follows that $A$ consists exactly of the elements of $B$ of order greater than $2$ together with the center of $B$ and hence $A$ is characteristic in $B$.

Let $S_1,S_2\in 2^A_{*Small}$ and let $\Gamma_1=\Cay(A,S_1)$ and $\Gamma_2=\Cay(A,S_2)$.  Suppose that $\Gamma_1\cong\Gamma_2$ and let $\varphi$ be a graph isomorphism from $\Gamma_1$ to $\Gamma_2$. Note that $\varphi$ induces a group isomorphism from $\Aut(\Gamma_1)=B$ to $\Aut(\Gamma_2)=B$ and hence $\varphi\in\Aut(B)$. Since $A$ is characteristic in $B$, $\varphi\in\Aut(A)$ and $S_1$ and $S_2$ are conjugate via an element of $\Aut(A)$. This shows that $|\UAGRR(A)|\geq |2^A_{*Small}|/|\Aut(A)|$. Let $m$ be the number of elements of order at most $2$ of $A$. By Theorem~\ref{th:epsilon2}, we have $|2^A_{*Small}|\geq 2^{m/2+n/2}-2^{m/2+11n/24+2(\log_2(n))^2+2}$. Since $|\Aut(A)|\leq  2^{(\log_2(n))^2}$, it follows that
$$|\UAGRR(A)|\geq 2^{m/2+n/2-(\log_2(n))^2}-2^{m/2+11n/24+(\log_2(n))^2+2}.$$

Let $\UCGN(A)$ denote the set of unlabeled Cayley graphs on $A$ with automorphism group strictly greater than $B$. Note that
$$\frac{|\UAGRR(A)|}{|\UAGRR(A)|+|\UCGN(A)|}=1-\frac{|\UCGN(A)|}{|\UAGRR(A)|+|\UCGN(A)|}\geq 1-\frac{|\UCGN(A)|}{|\UAGRR(A)|}.$$
By Theorem~\ref{th:epsilon2}, we have $|\UCGN(A)|\leq 2^{m/2+11n/24+2(\log_2(n))^2+2}$ and thus
$$\frac{|\UCGN(A)|}{|\UAGRR(A)|}\to 0,$$
as $n\to\infty$. This completes the proof.
\end{proof}

\noindent\textbf{Acknowledgment:} We would like to thank the anonymous referees for their many helpful comments.

\end{document}